\documentclass[12pt]{amsart}
\usepackage[]{amsmath, amsthm, amsfonts, verbatim, amssymb, tikz}

\usepackage[backref=page]{hyperref}
\hypersetup{colorlinks = true} 
\usepackage{rotating}
\usepackage{epstopdf}
\usepackage{pifont}
\usepackage{enumerate}
\usepackage{graphicx}

\DeclareMathAlphabet{\mathpzc}{OT1}{pzc}{m}{it}
\usepackage{tikz,tikz-cd, color}
\usepackage[makeroom]{cancel}
\usetikzlibrary{arrows}
\usepackage{lipsum} 
\renewcommand*{\backref}[1]{}
\renewcommand*{\backrefalt}[4]{%
	\ifcase #1 (Not cited.)%
	\or        (Cited on page~#2.)%
	\else      (Cited on pages~#2.)%
	\fi}

\newtheorem{theorem}{Theorem}[section]
\newtheorem*{theorem*}{Theorem}
\newtheorem{lemma}[theorem]{Lemma}
\newtheorem{proposition}[theorem]{Proposition}

\newtheorem{definition}[theorem]{Definition}

\newtheorem{question}[theorem]{Question}

\theoremstyle{remark}
\newtheorem{remark}[theorem]{Remark}
\newtheorem{example}[theorem]{Example}

\newcommand{\bF}{\mathbb{F}}

\newcommand{\bQ}{\mathbb{Q}}
\newcommand{\bP}{\mathbb{P}}
\newcommand{\bR}{\mathbb{R}}
\newcommand{\bZ}{\mathbb{Z}}

\newcommand{\cO}{\mathcal{O}}

\newcommand{\Exc}{{\rm Exc}}

\newcommand{\NE}{{\overline{\rm NE}}}

\newcommand{\Proj}{{\bf Proj}}

\newcommand{\Sing}{{\rm Sing}}

\newcommand{\Supp}{{\rm Supp}}

\newcommand{\Vol}{{\rm Vol}}

\def\<{\langle}
\def\>{\rangle}

\def\ni{\noindent}
\def\ra{\rightarrow}
\def\dra{\dashrightarrow}

\address[]{Department of Mathematics, National Cheng Kung University, Tainan 70101, Taiwan
}
\email{cjlai72@mail.ncku.edu.tw}
\begin{document}
	\title[Anticanonical volumes of weak $\bQ$-Fano threefolds]{On anticanonical volumes of weak $\bQ$-Fano terminal threefolds of Picard rank two}
	\author{Ching-Jui\ Lai}
	\begin{abstract} We show that for a weak $\bQ$-Fano threefold $X$ of Picard rank two ($\bQ$-factorial with at worst terminal singularities), the anticanonical volume satisfies $-K_X^3\leq 72$ except in one case, and the equality holds only if $X=\bP_{\bP^2}(\cO_{\bP^2}\oplus\cO_{\bP^2}(3))$. The approach in this article can be served as a general strategy to establish the optimal upper bound of $-K_X^3$ for any canonical Fano threefolds, where the described main result serves as the first step.
\end{abstract}
	\maketitle
\section{Introduction} Through this article, we work over an algebraically closed field of characteristic zero. A Fano variety is a normal projective variety with an ample anticanonical divisor. From the minimal model conjecture, Fano varieties are building blocks for varieties with negative Kodaira dimension. We recall two main classes of singularities arose from the minimal model program.  

\begin{definition} A normal variety $X$ has canonical (resp. terminal) singularities if $K_X$ is $\bQ$-Cartier and for any proper birational morphism $\mu:\tilde{X}\ra X$ from a normal variety $\tilde{X}$, where one can write $K_{\tilde{X}}\equiv\mu^*K_X+\sum a_iE_i$ and $E_i$ are $\mu$-exceptional divisors, the discrepancies $a_i\geq0$ (resp. $a_i>0$) for all $i$.
\end{definition}

The two main classes of varieties considered in this work are the class of canonical Fano varieties and the class of weak $\bQ$-Fano varieties. Varieties in the latter class have weaker positivity but are less singular.
\begin{definition} A canonical (weak) Fano variety is a normal projective variety $X$ with ample (resp. nef and big) $\bQ$-Cartier anticanonical divisor $-K_X$ and at worst canonical singularities. 

A weak $\bQ$-Fano variety is a $\bQ$-factorial normal projective variety $X$ with nef and big anticanonical divisor $-K_X$ and at worst terminal singularities. 
\end{definition}

From \cite{HM, BCHM}, one can always take a terminalization and then a $\bQ$-factorization of a canonical Fano variety $X'$, from which we then obtain a weak $\bQ$-Fano variety $X$ with a birational morphism $\pi:X\ra X'$ such that $\pi^*K_{X'}=K_X$. In this case, $X'=\Proj(\oplus_{m\geq0} H^0(X,\cO_X(-mK_X))$ is the \emph{anticanonical model} of $X$. In particular, the anticanonical volume remains unchanged: $\Vol(-K_{X'})=-K_{X'}^{\dim X'}=-K_X^{\dim X}=\Vol(-K_X)$. From recent results \cite{Bir1, Bir2} on the general boundedness of Fano varieties, known as  Borisov-Alexeev-Borisov conjecture, the anticanonical volume for canonical Fano varieties of a fixed dimension is bounded both from below and above. It is then very natural and interesting to ask what the optimal bounds are. From the above discussion, it is enough to work in the class of weak $\bQ$-Fano varieties.

We focus on $\dim X=3$ in this article, where the explicit geometry of terminal threefolds has been intensely studied for decades. It was established much earlier from \cite{KMMT} the existence of lower and upper bounds of $-K_X^3$: For any weak $\bQ$-Fano threefold $X$, one has 
$$\frac{1}{(24!)^3}\leq-K_X^3\leq 6^3\cdot(24!)^2.$$ 
However, these inequalities are far from optimal. If $X$ is smooth, it is known that $-K_X^3\leq 64$ from the classification of smooth Fano threefolds. 
\begin{question} Find the optimal lower and upper bound of anticanonical volumes for weak $\bQ$-Fano threefolds.
\end{question}
The problem becomes much harder due to the appearance of singularities and weakened positivity of $-K_X$ on a terminalization. Surprisingly, combined with Reid's Riemann-Roch formula for terminal threefolds, the optimal lower bound is established in \cite{CC} by their method of ``unpacking the baskets.''
\begin{theorem}[\cite{CC}]\label{lb} For a weak $\bQ$-Fano threefold $X$, the anticanonical volume satisfies $-K_X^3\geq1/330$. This bound is optimal for weighted hypersurfaces $X_{66}\subseteq\bP(1,5,6,22,33)$, cf. \cite{IF}.
\end{theorem}

However, finding the optimal upper bound is more delicate. Early results require the restriction to Gorenstein singularities or having the ampleness condition of $-K_X$.
\begin{theorem}\label{ub} Let $X$ be a weak $\bQ$-Fano threefold ($\bQ$-factorial with at worst terminal singularities). 
	\begin{enumerate}[$(i)$]
		\item If $X$ is Gorenstein and Fano, then $X$ is smoothable and $-K_X^3\leq 64,$ cf. \cite{Nam}. 
		\item If $X$ is Gorenstein weak $\bQ$-Fano, then $-K_X^3\leq 72$. Moreover, $-K_X^3=72$ if and only if the anticanonical model of $X$ is $\bP(1,1,1,3)$ or $\bP(1,1,4,6)$, cf. \cite{Pr1}.
		\item If $X$ is not Gorenstein but $\rho(X)=1$, then $-K_X^3\leq 125/2$. Moreover, $-K_X^3=125/2$ if and only if $X\cong\bP(1,1,1,2)$, cf. \cite{Pr2}.
	\end{enumerate}
\end{theorem}
The upper bound $-K_X^3\leq 72$ is conjectured for Gorenstein case by Iskovskih and proved in \cite{Pr1}. However, find the optimal upper bound for weak $\bQ$-Fano threefolds of Picard rank $\rho(X)\geq2$ is only partially solved. In the Gorenstein case, one can study the geometry of pluri-anticanonical systems, where the computation remains the same as for smooth cases, cf. \cite{Isk}. However, the study of the pluri-anticanonical systems in general is more complicated and only gets its attention recently, cf. \cite{CJ}. For $\bQ$-Gorenstien of Picard rank one case, the work \cite{Pr2} utilizes the index bound of Suzuki \cite{Su} and Kawamata's computation on Chern classes \cite{Ka2}, where the condition $\rho(X)=1$ is indispensable. Of course, in both cases, we use Reid's Riemann-Roch for terminal threefolds, and by far this is almost the only applicable tool for the remaining unsolved cases. 

Without Gorenstein condition, we restrict to $\rho(X)=2$ in this article and show that 72 is the optimal upper bound of anticanonical volume except in one case. Roughly speaking, by running two-ray games associated with the two extremal rays of $X$, we end up in two cases:
\begin{enumerate}
	\item[({\bf I})] either there exist Mori fiber spaces $X\stackrel{\chi_l}{\dashrightarrow}X_l\rightarrow Z_l$ and $X\stackrel{\chi_r}{\dashrightarrow}X_r\rightarrow Z_r$, or 
	\item[({\bf II})] $X_l$ admits a Mori fiber space $X\stackrel{\chi_l}{\dashrightarrow}X_l\rightarrow Z_l$ and a $K$-trivial divisorial contraction $\varphi_r:X\rightarrow Y_r$ with the exceptional divisor $E$.
\end{enumerate} 
See Section \ref{pre} for the notation and construction.
\theoremstyle{theorem}
\newtheorem*{main}{Main Theorem}
\begin{main}\label{main} For $X$ a weak $\bQ$-Fano threefold of $\rho(X)=2$, there are two cases:
\begin{enumerate}
	\item[$({\bf I})$] If both two-ray games end in Mori fibre spaces, then $-K_X^3\leq 54$.
	\item[$({\bf II})$] If a $K$-trivial divisorial contraction appears in one of the two-ray games, then $-K_X^3\leq 72$, unless $\dim Z_l=1$ and $\dim\varphi_r(E)=0$. Moreover, equality holds only if $X=\bP_{\bP^2}(\cO_{\bP^2}\oplus\cO_{\bP^2}(3))$.
	\end{enumerate}
In $(${\bf II}$)$, when $\dim Z_l=1$ and $\dim\varphi_r(E)=0$, we have $-K_X^3\leq 81.$	
\end{main}
\begin{remark} The exceptional case in ({\bf II}), i.e., when $\dim Z_l=1$ and $\dim\varphi_r(E)=0$, will be treated in a forthcoming paper, and we expect that the estimation $-K_X^3\leq72$ remains to hold.   
\end{remark}

From Theorem \ref{ub} and our main theorem, we suspect the upper bound 72 is optimal for all cases.
\theoremstyle{theorem}
\newtheorem*{main2}{Main Question}
\begin{main2} For $X$ a weak $\bQ$-Fano threefold, the anticanonical volume satisfies $-K_X^3\leq 72.$  
\end{main2}

Our first strategy to attack the Main Question is to utilize the geometry of all possible birational models of $X$ after running two ray games along with extremal rays of $X$. This works if $\rho(X)=2$.

If $\rho(X)=2$, then there are only two birational models and $-K_X$ can be expressed as a sum of two divisors coming from these two models. The simplest case is when $X$ equips simultaneously a del Pezzo fibration $\varphi:X\ra\bP^1$ and a conic bundle $\psi:X\ra Z$. Then $-K_X\equiv aF+bH$, where $F$ is a general fibre of $\varphi$ and $H$ is the pullback of some ample divisor on $Z$. It follows easily that $-K_X^3\leq 54$ from Lemma \ref{key}, \ref{Z2}, and the length of extremal curves on del Pezzo surfaces. Two problems appear in this approach. The minor one is that one may obtain a Mori fibre space only after a non-trivial small birational map. If there are two Mori fibre spaces birational to $X$, then this is solved in Section \ref{I} as the case ({\bf I}). The more severe problem is the existence of a $K$-trivial extremal divisorial contraction contracting a divisor $E$ on a birational model $X'$ of $X$. In this case, the singularities of $(X',E)$ is not well understood, cf. Remark \ref{tech}. This case is treated in Section \ref{II} as the case ({\bf II}).
		
When $\rho(X)>2$, there is no simple expression of $-K_X$ in terms of divisors from the birational models of $X$ even if all the extremal rays lead to Mori fiber spaces. The second strategy to attack the Main Question is then to examine case by case the geometry of extremal contractions of $X$. For example, if $\varphi_R$ is of fiber type over $\bP^1$, then $\rho(X)=2$ and we have shown $-K_X^3\leq 72$. We briefly describe how then to bound the anticanonical volume in this approach, where our Main Theorem serves as a successful first step. We also state some related problems along the way. 

Start with a $K_X$-negative extremal contraction $\varphi_R:X\ra X'$ of a weak $\bQ$-Fano threefold $X$, which always exists in our setting. We may assume now $\varphi_R$ is of fiber type over a surface $Z$, then our result again says that $-K_X^3\leq72$ provided $\rho(X)=2$ (equivalently $\rho(Z)=1$). So the next step is to drop the Picard rank condition. 
\begin{question} Let $X$ be a weak $\bQ$-Fano threefold. If one of its $K_X$-negative extremal contractions $\varphi_R:X\ra Z$ is a conic bundle, then $-K_X^3\leq72$.
\end{question}
If $\varphi_R$ is a divisorial contraction to a point, then $X'$ remains a weak $\bQ$-Fano threefold with $-K_{X'}^3\geq-K_X^3$. We replace $X$ by $X'$ and proceed. If $\varphi_R$ is divisorial to a curve, then $X'$ is not necessarily weak Fano. If $X'$ is weak $\bQ$-Fano, then again we replace $X$ by $X'$ and proceed. Hence we may assume that any $K$-negative extremal contraction is a divisorial contraction to curve or small. 
\begin{question} Are there examples of weak $\bQ$-Fano threefolds whose $K$-negative extremal contractions are either divisorial to curves or small? Can we show $-K_X^3\leq72$ in this case?
	\end{question} 
Finally, there is also a chance by investigating the geometry of the pluri-anticanonical systems of $X$. It is clear that much more has to be done in the future in solving Main Question. 

The last remark is that the reduction of computing anticanonical volumes to the class of weak $\bQ$-Fano varieties also apply to the class of canonical weak Fano varieties. Hence our strategies and results extend freely to this class of varieties. 

Here is the organization of this paper: In Section \ref{aux}, we present relevant known results and prepare some crucial lemmas. In Section \ref{pre}, by running two ray games, it is shown that we can split the study into two cases ${\bf (I)}$ and ${\bf (II)}$. In Section \ref{I}, we deal with case ${\bf (I)}$, which essentially corresponding to double Mori fiber spaces. In Section \ref{II}, we consider the case ${\bf (II)}$, where one of the extremal rays is $K$-trivial and divisorial. In Section \ref{opt}, we classify weak $\bQ$-Fano threefolds of Picard rank two with optimal degree 72 and hence complete the proof of Main Theorem. 

\vskip 0.2 cm

\ni\textbf{Acknowledgements.} The author is grateful for many inspiring discussions with Jungkai A.  Chen and Jheng-Jie Chen. The author also thanks Yuri Prokhorov for answering his questions and Shin-Yao Jow for consultation on Mori dream spaces. The author is grateful to Research Institute for Mathematical Sciences in Japan where this work was partially done. This work is supported by National Center of Theoretical Sciences in Taiwan and funded by MOST 106-2115-M-006-019. Finally, the author thanks the referees for their useful comments and suggestions. 

\section{Auxiliary results}\label{aux} A general reference for the minimal model program is \cite{KM}. On the other hand, we include here several relevant results for estimating the anticanonical volumes. The first statement is the Negativity lemma. 
\begin{lemma}\label{neg} Let $\chi:X_1\dra X_2$ be an isomorphism in codimension one of normal $\bQ$-factorial varieties. Let $D_1\geq0$ a divisor on $X_1$ and $D_2$ be its proper transform on $X_2$. If $D_1$ is nef, then for a common resolution $p_i:W\ra X_i$ for $i=1,2$, the divisor $E:=p_2^*D_2-p_1^*D_1$ is effective and exceptional over $X_i$'s.
\end{lemma}
\begin{proof} See \cite{KM}.
\end{proof}

Next, we turn to the geometry of conic bundles. When there is a conic bundle $X\ra Z$ from a terminal projective threefold, the geometry of the surface $Z$ is quite restrictive. The following work is essential to us.
\begin{theorem}\label{conic}\cite{MP} Let $X$ be a $\bQ$-factorial terminal threefold with an extremal contraction $X\ra Z$. If $\dim Z=2$, then $Z$ has only singularities of type $A$.
\end{theorem}
As a consequence, we can do some relevant numerical computations on $Z$. The following proposition is a generalization of results for conic bundles in \cite{MM, Pr1}  and should be well-known to experts. We include the proof here for the convenience of the reader. 
\begin{proposition}\label{disc} Let $f:X\ra Z$ be a $K$-negative extremal contraction from a weak $\bQ$-Fano threefold of $\rho(X/Z)=1$ with $\dim Z=2$. Then $Z$ is $\bQ$-factorial by Theorem \ref{conic}. Furthermore,  
	\begin{enumerate}[$(a)$]
		\item $-4K_Z\equiv f_*K_X^2+\Delta$, where $\Delta$ is a reduced Weil divisor, and
		\item $Z$ is a weak del Pezzo surface and $\rho(X)\leq 10$.
	\end{enumerate} 
\end{proposition}
\begin{proof} Since $f:X\ra Z$ has only 1-dimensional fibres, $f$ is flat over $Z':=Z-f(\Sing(X))\cup\Sing(Z)$, cf. \cite{MM}. In particular, the discriminant divisor $\Delta_f$ defined over $Z'$ is a simple normal crossing divisor and we take $\Delta=\overline{\Delta_f}$ on $Z$. 

To prove the first statement, we observe first that $X$ has isolated singularities. Take $B$ on $Z$ a general very ample smooth curve avoiding the discrete set $f(\Sing(X))\cup\Sing(Z)$ and intersecting transversely to $\Delta$, then $T:=X_B\ra B$ is a blow up of a $\bP^1$-bundle $T'\ra B$ as singular fibres of $T$ can only be the union of two different lines. Since $K_{T'/B}^2=0$ and $K_{T/B}=K_{X/Z}|_T=\pi^*K_{T'/B}+\sum_{i=1}^kE_i$, where $\pi:T\ra T'$ consists of smooth blow-ups at distinct points $p_1,\dots,p_k$ and $E_i$'s are the corresponding $(-1)$-curves on $T$, $-K_{T/B}^2=\sum_{i=1}^k(-E_i^2)$ is rational equivalence to the sum of the singular points of fibres with all coefficients 1. Hence 
	$$-f_*(K_{X/Z}^2)\cdot B=-f_*(K_{X/Z}^2\cdot f^*B)=-f_*(K_{T/B}^2)=\Delta\cdot B.$$ 
On the other hand, consider $f_*:A^2(X)\ra A^2(Z)\cong\bZ[Z]$ the proper push-forward. Then $f_*K_X=m[Z]$ as 2-cycles on $Z$ for some $m\in\bZ$. By using the projection formula for the flat pull-back $f^*:A^*(Z')\ra A^*(f^{-1}(Z'))$, we have $m=m[Z]|_{Z'}\cdot[{\rm pt}]=m[Z']\cdot[{\rm pt}]=-K_{f^{-1}(Z')}\cdot l=-K_X.l=2$ for a general fibre $l$ of $f:X\ra Z$. This proves the first statement.  
	
Since $|m(-K_X)|$ defines a birational morphism for \mbox{$m\gg0$} by base point freeness theorem \cite{KM}, $L=f_*(K_X^2)$ is movable and big on $Z$. It follows that $-K_Z$ is big by $(a)$. We now show that $-K_Z$ is nef. Suppose that $K_Z\cdot C>0$ for some integral curve $C$ on $Z$, then $(4K_Z+\Delta)\cdot C=-L\cdot C\leq0$ as $L$ is movable. In particular, 
$$\Delta\cdot C=-L\cdot C-4K_Z\cdot C<0,$$ 
and $C$ must be a component of $\Supp(\Delta)$. It follows that $$C^2=\Delta\cdot C-(\Delta-C)\cdot C<0.$$ 
	
Let $g:\tilde{Z}\ra Z$ be the minimal resolution of $Z$ and write $g^*C=\tilde{C}+E$ for some exceptional $E\geq0$. Since $g^*K_Z=K_{\tilde{Z}}$ by Theorem \ref{conic}, we get 
	$$(K_Z+C)\cdot C=(K_{\tilde{Z}}+\tilde{C})\cdot g^*C=2p_a(\tilde{C})-2+E\cdot \tilde{C}\geq-2.$$
We consider two possibilities. If $C\subseteq\Supp(\Delta)$, then as $\Delta$ is reduced, 
	$$-2\leq (K_Z+C)\cdot C\leq (K_Z+\Delta)\cdot C=-(3K_Z+L)\cdot C\leq-3K_Z\cdot C.$$
This is absurd since $K_Z$ is Cartier. If $C\nsubseteq\Supp(\Delta)$, then 
	$$-2\leq (K_Z+C)\cdot C\leq(K_Z+\Delta+C)\cdot C=-(3K_Z+L)\cdot C+C^2<-3K_Z\cdot C.$$
We still get a contradiction as above. 
\end{proof}	

\section{Running two-ray games}\label{pre}
\subsection{Construction} Let $X$ be a weak $\bQ$-Fano threefold of Picard rank $\rho(X)=2$. Then the Mori cone of $X$ is $\NE(X)=\bR_{\geq0}[\Gamma_1]+\bR_{\geq0}[\Gamma_2]$ for some rational curves $\Gamma_i$'s.  Moreover, there exists extremal contractions $\varphi_i:X\ra Y_i$ for each extremal ray $R_i:=\bR_{\geq0}[\Gamma_i]$. There are three possible types of the morphism $\varphi_i$: divisorial, small, or fibre type. 
\begin{lemma}\label{vol} The anticanonical volume is nondecreasing for each step of a minimal model program of birational type.
\end{lemma}
\begin{proof} If a birational map $X\dra X'$ is a step of a minimal model program, then for a common resolution $p:W\ra X$ and $q:W\ra X'$, we have by negativity lemma that $p^*K_{X}=q^*{K_X'}+E$ for some $E\geq0$. In particular, 
\begin{align*}\Vol(-K_{X'})=\Vol(p^*(-K_X)+E)
					\geq\Vol(p^*(-K_X))=\Vol(-K_X).\end{align*}
\end{proof}
\begin{definition} Let $\varphi:X\ra Y$ be an extremal contraction. We say $\varphi$ is $K$-trivial if $K_X\cdot C=0$ for any curve contracted by $\varphi.$
	\end{definition}
\begin{lemma}\label{qfac} If $\varphi:X\ra Y$ is a $K$-trivial divisorial extremal contraction of $\rho(X/Y)=1$, then $Y$ is $\bQ$-factorial.
\end{lemma}
\begin{proof} Since $\rho(X/Y)=1$ and $\Exc(\varphi)$ contains an irreducible divisor $E$, it is easy to see that $E.\Gamma<0$ for $\Gamma$ a contracted curve of $\varphi$ and $\varphi$ is a $(K_X+\epsilon E)$-negative extremal contraction for $0<\epsilon\ll1$. In particular, as $(X,\epsilon E)$ is klt, by contraction theorem we know that $\Exc(\varphi)=E$ and $Y$ is $\bQ$-factorial.
\end{proof}

Finally, since $X$ is rationally connected by \cite{HM:RCC, Zh} and hence so is any birational model of $X$, by three dimensional abundance \cite{Miy1, Miy2, Ka} each birational model of $X$ has a $K$-negative extremal ray.

For each possible type of the ($K$-non-positive) extremal contractions $\varphi_i$'s, we consider the following: 
\vskip 0.4 cm
\begin{enumerate}
	\item[${\bf (D_-)}$] {\bf $\varphi_i$ is divisorial and $K$-negative:} Then $Y_i$ is a terminal $\bQ$-Fano threefold as $\rho(Y_i)=1$. By Theorem \ref{ub} and Lemma \ref{vol}, $-K_X^3\leq -K_{Y_i}^3\leq 64.$
\vskip 0.4 cm	
	\item[${\bf (D_0)}$] {\bf $\varphi_i$ is divisorial and $K$-trivial:} From Lemma \ref{qfac}, the resulting variety is $\bQ$-factorial, and it acquires \emph{canonical} singularities.  
\vskip 0.4 cm
	\item[${\bf (S)}$] {\bf $\varphi_i$ is small:} After a flip or flop $X\dra X^+$, there is still a $K$-negative extremal ray of $X^+$ from the discussion of last paragraph. If a divisorial contraction occurs, then we get $-K_X^3\leq64$ by Lemma \ref{vol} as in case ${\bf (D_-)}$. If we have a fibre type contraction, then we get a Mori fibre space. Otherwise, this is a flipping contraction and we take a flip. Continue the same process by considering a $K$-negative extremal contraction, this leads to a sequence of flips, which terminates by the termination of flips in dimension three and the last step must be a $K$-negative divisorial contraction or a Mori fibre space. In the former case, we get $-K_X^3\leq64$ as before. Hence we only need to consider the latter case, where this  two ray game  terminates in a Mori fibre space. 
\vskip 0.4 cm
	\item[${\bf (F)}$] {\bf $\varphi_i$ is of fibre type:} We get a Mori fibre space. 
\end{enumerate}
\vskip 0.4 cm
In conclusion, start with each $\varphi_i$ and run a two ray game, either ${\bf (D_-)}$ appears in some step and we conclude $-K_X^3\leq 64$, or after re-index $\varphi_i$'s to $\varphi_\star$, where $\star\in\{l,r\}$ stands for left and right, we end up with the following two possible diagrams ${\bf (I)}$ and ${\bf (II)}$ depending on whether ${\bf (D_0)}$ appears.

\begin{itemize}
	\item[${\bf (I)}$]\emph{ None of $\varphi_i$'s is of type ${\bf (D_0)}$.} We have a diagram:
	
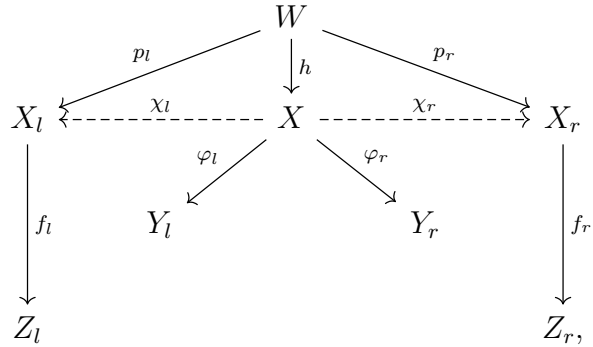
\begin{center}
		\begin{tikzcd}
			&& W\arrow[dll,"p_l", swap]\arrow[drr,"p_r"]\arrow[d,"h"]  & &\\
			X_l\arrow[dd,"f_l"] && X\arrow[rd,"\varphi_r"]\arrow[ld,swap,"\varphi_l"]\arrow[ll,dashed,swap,"\chi_l"]\arrow[rr,dashed,"\chi_r"] &  &X_r\arrow[dd,"f_r"]\\
			      &Y_l&&Y_r& \\
			Z_l   &&& & Z_r, 
		\end{tikzcd} 
	\end{center}
	where 
	\begin{enumerate}[$(i)$]
		\item $\chi_\star$ is either identity or a small birational map; 
		\item $f_\star$ is a Mori fiber space. Note that it is possible here $\chi_l$ is identity and $\varphi_l=f_l$, and similarly for $\chi_r$;
		\item $p_\star$ and $h$ are smooth resolutions of $X_\star$ and $X$, and hence birational.
	\end{enumerate}
\vskip 0.5 cm
	\item[${\bf (II)}$] \emph{One of $\varphi_*$'s (at most one), say $\varphi_r$, is of type ${\bf (D_0)}$.} We have a diagram:
		\begin{center}
		\begin{tikzcd}
			&&& W\arrow[dlll,"p_l", swap]\arrow[d,"h"] & \\
			X_l\arrow[d,"f_l"] &&& X\arrow[dll,swap,"\varphi_l"]\arrow[lll,dashed,swap,"\chi_l"]\arrow[r,"\varphi_r"]   &Y_r\\
			Z_l   &Y_l&&E\arrow[u,hook]\arrow[r]&\varphi_r(E)\arrow[u,hook],
		\end{tikzcd} 
		\end{center}
	where 	
	\begin{enumerate}[(i)]
		\item $\chi_l$ is either identity or a small birational map; 
		\item $f_l$ is a Mori fiber space. Note that it is possible here that $\chi_l$ is identity and $\varphi_l=f_l$;
		\item $p_l$ and $h$ are smooth resolutions of $X_l$ and $X$, and hence birational.		
	\end{enumerate}
\end{itemize}
Note that as the morphism $f_\star$ is a Mori fiber space and $\rho(X_\star)=\rho(X)=2$, the base of the resulted Mori fibre spaces can only have $\dim Z_\star=1$ or $2$.

\subsection{Notation} If $\dim Z_\star=1$, then $Z_\star=\bP^1$ and we set $F_\star$ on $X_\star$ to be a general fiber of $f_\star:X_\star\ra Z_\star$. If $\dim Z_\star=2$, then $Z_\star$ is a del Pezzo surface of Picard rank one and we set $H_\star$ on $X_\star$ to be the pull-back of an ample Cartier divisor on $Z_\star$. We will make a choice of $H_\star$ later on. We abuse the notation and write $H_\star$ and $F_\star$ again for their proper transforms on $X$. In Section 5, we also use $H_X$ and $F_X$ for the proper transforms on $X$. The notation should be clear from the context.

\section{Proof of Main Theorem: Case ${\bf (I)}$}\label{I}
Recall that we have the diagram, 
	\begin{center}
		\begin{tikzcd}
			&& W\arrow[dll,"p_l", swap]\arrow[drr,"p_r"]\arrow[d,"h"]  & &\\
X_l\arrow[d,"f_l"] && X\arrow[ll,dashed,swap,"\chi_l"]\arrow[rr,dashed,"\chi_r"] &  &X_r\arrow[d,"f_r"]\\
			Z_l   &&& & Z_r, 
		\end{tikzcd} 
	\end{center}
where the morphism $f_\star:X_\star\ra Z_\star$ is always a Mori fibre space. The estimation of the anticanonical volume in this case is broken down into the following lemmas. 

\begin{lemma}\label{diff} Let $A_\star$ on $X_\star$ be $F_\star$ or $H_\star$ according to $\dim Z_\star=1$ or $2.$ Then the divisors $A_l$ and $A_r$ on $X$ are independent in $N^1(X)_\bR$.
\end{lemma}
\begin{proof} Since $X$ is a Mori dream space by \cite{BCHM}, by the general theory of Mori dream spaces \cite{HK} the two divisors $A_l$ and $A_r$ on $X$ generate two extremal rays of $\overline{{\rm Mov}}(X)$ and hence are independent.  
	\end{proof}	

\begin{lemma}\label{key} We keep the same notation as in Lemma \ref{diff} and write $-K_X\equiv a_lA_l+a_rA_r$. Then 
	\begin{enumerate}[$(a)$]
		\item the coefficient $a_r\in\bQ$ satisfies $0<a_r\leq -K_{X_l}\cdot C_l$, and similarly for $a_l$; 
		\item there is an inequality, $$-K_X^3\leq a_lK_{X_l}^2\cdot A_l+a_rK_{X_r}^2\cdot A_r.$$
	\end{enumerate}
\end{lemma}
\begin{proof} Write $-K_X-a_lA_l\equiv a_rA_r$. Since $\kappa(X,-K_X)=3$ but $\kappa(X,a_rA_r)<3$, we must have $a_l>0$. Similarly, we get $a_r>0$. Push forward this relation to $X_l$ and intersecting with a fibre curve $C_l$ of $f_l:X_l\ra Z_l$, we get $a_r=\frac{-K_X\cdot C_l}{A_r\cdot C_l}\in\bQ$. As $A_r$ is integral and $C_l$ moves in a covering family, we have $A_r\cdot C_l\in\bZ_{>0}$ and hence $(a)$. 
	
Since for $\star\in\{l,r\}$ the divisors $A_\star$ being the pullback of an ample divisor from $Z_\star$ is nef on $X_\star$ and $-K_X$ is nef, by Lemma \ref{neg} we can write, 
$$\begin{cases} (p_\star)^*A_\star=h^*A_\star-E_\star&,\ {\rm where}\ E_\star\geq0\ {\rm is\ exceptional};\\ 
 h^*(-K_X)=(p_\star)^*(-K_{X_\star})-G_\star&,\ {\rm where}\ G_\star\geq0\ {\rm is\ exceptional.}
	\end{cases}$$
	
From $-K_X\equiv a_lA_l+a_rA_r$, we have 
	\begin{align*} h^*(-K_X)\equiv a_lh^*A_l+a_rh^*A_r=a_lh^*A_l+a_rp_r^*A_r+a_rE_r ,
	\end{align*}
and hence 
	\begin{align*} -K_X^3&=h^*(-K_X)^3\\
	&=a_lh^*A_l\cdot h^*(-K_X)^2+a_rp_r^*A_r\cdot h^*(-K_X)^2+a_rE_r\cdot h^*(-K_X)^2 \\
	&=a_lh^*A_l\cdot h^*(-K_X)^2+a_rp_r^*A_r\cdot (p_r^*(-K_{X_r})-G_r)^2 \\
	&=a_lh^*A_l\cdot h^*(-K_X)^2+a_r(K_{X_r}^2\cdot A_r-2p_r^*A_r\cdot p_r^*(-K_{X_r})\cdot G_r+p_r^*A_r\cdot G_r^2) \\
	&\leq a_lh^*A_l\cdot p^*(-K_X)^2+a_rK_{X_r}^2\cdot A_r\\
	&\leq a_lK_{X_l}^2\cdot A_l+a_rK_{X_r}^2\cdot A_r
	\end{align*} 
where we have used $E_r\cdot h^*(-K_X)^2=0=p_r^*A_r\cdot p_r^*(-K_{X_r})\cdot G_r$, and $p_r^*A_r\cdot G_r^2=(G_r|_{p_r^*A_r})^2\leq0$. Indeed, take $S_r\in |p_r^*(mA_r)|$ for $m\gg0$ a general smooth surface of this base point free linear system, then $G_r|_{S_r}$ is an exceptional curve for the birational morphism $p_r|_{S_r}$ and hence $m^2p_r^*A_r\cdot G_r^2=(G_r|_{S_r})^2\leq0$ by \cite[Corollary 2.7]{B}. Similar reasoning applies to $h^*A_l\cdot h^*(-K_X)^2$ and hence the last inequality. This completes the proof of $(b)$.
\end{proof}

Let $Z$ be one of $Z_l$ and $Z_r$, and let $K, H$ and $F$ be the corresponding divisors on $X_\star$. 
\begin{lemma}\label{Z2} If $\dim Z=1$, then $K^2\cdot F=K_F^2\leq 9$. If $\dim Z=2$, then one can choose $H$ on $Z$ so that $K^2\cdot H\leq 12.$
\end{lemma}
\begin{proof} The result is clear when $\dim Z=1$. We assume $\dim Z=2$ and that $H=f^*C$ for some curve $C$ on $Z$. Then from Proposition \ref{disc} and its proof, we have $f_*K_X^2=-4K_Z-\Delta$ for a reduced divisor $\Delta$ and $-(K_Z+C)\cdot C\leq 2$. Since $Z$ is del Pezzo, we can choose $C$ to be an extremal curve of minimal length $-K_Z\cdot C\leq3$. There are two cases to consider. If $C\nsubseteq\Supp(\Delta)$, then 
	\begin{align*}K_X^2\cdot H=f_*K_X^2\cdot C=4(-K_Z\cdot C)-\Delta\cdot C \leq 4(-K_Z\cdot C)\leq 12.
	\end{align*}
	If $C\subseteq\Supp(\Delta)$, then 
	\begin{align*}K_X^2\cdot H=f_*K_X^2\cdot C=3(-K_Z\cdot C)-(K_Z+C)\cdot C-C\cdot (\Delta-C)\leq 11. 
	\end{align*}
	\end{proof}
	
\begin{proof}(of Main Theorem in Case ${\bf (I)}$) From Lemma \ref{key} $(a)$ and length of extremal ray on del Pezzo surfaces, we have 
	$$ 0<a_r\leq -K_{X_l}\cdot C_l\leq\begin{cases} 3,&\ {\rm if}\ \dim Z_l=1 \\ 2,&\ {\rm if}\ \dim Z_l=2\end{cases},
	$$
and similarly for $a_l\in\bQ_{>0}$. Hence from Lemma \ref{key} $(b)$, we have an upper bound of the  anticanonical volume, 
\begin{align*} -K_X^3&\leq a_lK_{X_l}^2\cdot A_l+a_rK_{X_r}^2\cdot A_r\\ 
&\leq \begin{cases} 3\cdot 9+3\cdot 9=54&,\ {\rm if}\ \dim Z_l=\dim Z_r=1\\
3\cdot 12+2\cdot 9=54&,\ {\rm if}\ \dim Z_l\neq\dim Z_r\\
2\cdot 12+2\cdot 12=48&,\ {\rm if}\ \dim Z_l=\dim Z_r=2\\
\end{cases}.
\end{align*}
\end{proof}

\section{Proof of Main Theorem: Case ${\bf (II)}$}\label{II} 
Recall that there is a diagram, 
\begin{center}
\begin{tikzcd}
			&&& W\arrow[dlll,"p_l", swap]\arrow[d,"h"] & \\
			X_l\arrow[d,"f_l"] &&& X\arrow[dll,swap,"\varphi_l"]\arrow[lll,dashed,swap,"\chi_l"]\arrow[r,"\varphi_r"]   &Y_r\\
			Z_l   &Y_l&&E\arrow[u,hook]\arrow[r]&\varphi_r(E)\arrow[u,hook].
		\end{tikzcd} 
		\end{center}
The following lemma is crucial to our main result.  
\begin{lemma}\label{adj} Suppose that $\dim \varphi_r(E)=0$. Then there is a curve $\Gamma$ on  $E$ moving in a one-dimensional family such that $0<-E\cdot \Gamma\leq3.$
\end{lemma}

\begin{proof} From \cite[Proposition 4.5]{Kol}, there is a subadjunction formula,
	\begin{align*}(K_X+E)|_{E^{\nu}}=K_{E^\nu}+{\rm Diff(0)},
	\end{align*}
where $\Delta={\rm Diff(0)}$ is effective but may have coefficients bigger than one. Let $h:S\ra E^\nu$ be the minimal resolution of $(E^\nu,\Delta)$, which always exists for surfaces, then we have $K_S+\Delta_S=h^*(K_{E^\nu}+\Delta)$ for some $\Delta_S\geq0$. It is enough to find a curve $\Gamma_S$ on $S$ moving in a one-dimensional family with $-(K_S+\Delta_S)\cdot \Gamma_S\leq3$ for if $\Gamma=h_*\Gamma_S$, then 
$$-E\cdot \Gamma=-(K_X+E)\cdot h_*\Gamma_S=-(K_S+\Delta_S)\cdot \Gamma_S\leq3.$$	 

We first note that being an exceptional divisor over a canonical singularity with a zero-dimensional center, $E$ is rationally chain connected by \cite{HM:RCC} and hence so is $S$. Let $D$ be a $(-1)$-curve on $S$ and $\mu:S\ra S'$ be the corresponding contraction. If $D$ maps to a point on $E^\nu$, then $S\ra E^\nu$ factors through $S'$ and $(K_X+E)|_{S'}=K_{S'}+\mu_*\Delta_{S}$. Hence after replacing by $S'$, we may assume that $S\ra E^\nu$ contracts no $(-1)$-curves of $S$. 

If $S$ contains no $(-1)$-curve, then $S=\bP^2$ or $S=\bF_n$ for some $n\neq1$\footnote{or $S$ has a $\bP^1$-bundle structure $S\rightarrow C$. In the proof below, replace $\bF_n\ra\bP^1$ by $S\ra C$.}. 
In the former case, $-(K_S+\Delta_S)\cdot \Gamma_S\leq 3$ by choosing $\Gamma_S$ to be a line in $\bP^2$. 
In the latter case, let $\Gamma_S$ be a general fiber of $\bF_n\ra \bP^1$. Then 
$-(K_S+\Delta_S)\cdot \Gamma_S\leq -K_S\cdot \Gamma_S=2$ 
and $\Gamma_S$ does not map to a point on $E^\nu$ as it moves in a covering family.  

Suppose now there is a $(-1)$-curve $C$ on $S$. Inductively blowing down $(-1)$-curves, we get $\tilde{\mu}:S\ra S_{\rm min}$ with $S_{\rm min}=\bP^2$ or $S_{\rm min}=\bF_n$ for some $n\neq1$, and we may write 
$$-(K_S+\Delta_S)=-\tilde{\mu}^*(K_{S_{\rm min}}+\Delta_{S_{\rm min}})+G,$$ 
where $G$ is exceptional over $S_{\rm min}$. If $S_{\rm min}=\bP^2$, let $D$ be a line avoiding $\tilde{\mu}(G)$ and $\Gamma_S$ its proper transform on $S$, then 
$$-(K_S+\Delta_S)\cdot \Gamma_S=-(K_{S_{\rm min}}+\Delta_{S_{\rm min}})\cdot D\leq3.$$
If $S_{\rm min}=\bF_n$, let $D$ be a general fiber of $\bF_n\ra\bP^1$ avoiding $\tilde{\mu}(G)$ and $\Gamma_S$ its proper transform on $S$, then 
$$-(K_S+\Delta_S)\cdot \Gamma_S=-(K_{S_{\rm min}}+\Delta_{S_{\rm min}})\cdot D\leq2.$$
\end{proof}

The next lemma is the counterpart of Lemma \ref{key} in the proof of Case ${\bf (I)}$.
\begin{lemma}\label{gen} If $\dim Z_l=1$, then the divisors $F_X$ and $E$ are independent in $N^1(X)_\bR$. Moreover, we can write $-K_X\equiv aF_X+bE$ with $a,b\in\bQ_{>0}$ and $0< b\leq 3$. 
Similarly, when $\dim Z_l=2$ we can write $-K_X\equiv aH_X+bE$ with $a,b\in\bQ_{>0}$ and $0< b\leq 2$. In either cases, we have 
$$\begin{cases}0<a\leq2b,&\ {\rm if}\ \dim\varphi_r(E)=1\\ 0<a\leq 3b,&\ {\rm if}\ \dim\varphi_r(E)=0\end{cases}.$$ 
\end{lemma}
\begin{proof} Since $\chi_l$ is small, on $X$ we have $\kappa(F_X)=1$ and hence $F_X$ can not be proportional to $E$ as $\kappa(E)=0$. Write $-K_X\equiv aF_X+bE$ for some $a,b\in\bR$. Since $-K_X$ is big but $bE\equiv -K_X-aF_X$ is not big, we must have $a>0$. Now push forward to $X_l$ and intersecting with a fibre curve $C_l$ of $f_l:X_l\ra Z_l$, we get 
$$0<b\leq bE\cdot C_l=aF_l\cdot C_l+bE\cdot C_l=-K_{X_l}\cdot C_l\leq \begin{cases}3 ,&\ {\rm if}\ \dim Z_l=1\\  2,&\ {\rm if}\ \dim Z_l=2\end{cases},$$
where $E\cdot C_l\geq1$ as $E$ is integral and $C_l$ moves in a covering family. This also shows that $b=\frac{-K_{X_l}\cdot C_l}{E\cdot C_l}\in\bQ_{>0}$. Since $aF_X\equiv -K_X-bE$, intersecting with an ample Cartier divisor shows that $a\in\bQ_{>0}$ as well. 

Let $\Gamma$ be the generating extremal curve of $\varphi_r$, then we have 
$$0=-K_X\cdot \Gamma=aF_X\cdot \Gamma+bE\cdot \Gamma.$$ 
If $\dim\varphi_r(E)=1$, then $Y_r$ has a one-dimensional center of canonical singularity along $C=\varphi_r(E)$. Since $X$ is smooth generically along $E$, $(X_{\eta_C},E_{\eta_C})$ is a resolution of a rational double point at the point $\eta_C\in Y_{\eta_C}$. This implies that $\eta_C\in Y_{\eta_C}$ is a $A_1$-singularity and hence 
\begin{align*}aF_X\cdot \Gamma=b(-E\cdot \Gamma)=-b(K_X+E)\cdot \Gamma=-bK_E\cdot \Gamma=2b.
\end{align*}
We claim that $F_X\cdot \Gamma\geq1$ and hence $a\leq 2b.$ 

For the proof of the claim, let $\Gamma$ be a moving curve as in Lemma \ref{adj} or a general fiber curve of \mbox{$\varphi_r|_E:E\rightarrow\varphi_r(E)$} when $\dim\varphi_r(E)=1$. Since $\Gamma$ moves in a one-dimensional family and $h(\Exc(W))$ has codimension two in $X$, by choosing $\Gamma$ not completely contained in $h(\Exc(W))$ we can find an integral curve $\Gamma_W\nsubseteq\Exc(W)$ on $W$ with $h_*\Gamma_W=\Gamma$. Say $\dim Z_l=1$. Note that $F_X=h_*p_l^*F$ as $\chi_l$ is small. Since the linear series $|F|$ is base point free so that $F$ is nef, it follows that $h^*h_*p_l^*F=p_l^*F+E'$ for an exceptional $\bQ$-divisor $E'\geq0$ and $$F_X\cdot\Gamma=h^*F_X\cdot\Gamma_W=(p_l^*F+E')\cdot\Gamma_W\geq F\cdot (p_l)_*\Gamma_W\geq1.$$
The same argument works when $\dim Z_l=2$ by replacing $F_X$ with $H_X$.

When $\dim\varphi_r(E)=0$, we apply Lemma \ref{adj} to get a curve $\Gamma$ on $X$ with
$$aF_X\cdot \Gamma=-b(K_X+E)\cdot \Gamma=-bE\cdot \Gamma\leq 3b.$$
This implies that $a\leq3b$ by the same argument as in the case of $\dim\varphi_r(E)=1.$ 
\end{proof}
We see that when $\dim\varphi_r(E)=0$, apply Lemma \ref{gen} and use $K_X\equiv\varphi_r^*K_{Y_r}$, we get 
$$a\leq aF_X\cdot \Gamma=b(-E\cdot \Gamma)\leq 3b\leq\begin{cases}9,&\ {\rm if}\ \dim Z_l=1\\  6, &\ {\rm if}\ \dim Z_l=2\end{cases}.$$

\begin{remark}\label{tech}	To estimate $a$ when $\dim\varphi_r(E)=0$ as in the last step of the above proof, one can try to argue in the following way. We get $-K_{Y_r}\equiv aF_{Y_r}$ by pushing forward to $Y_r$, where $F_{Y_r}:=(\varphi_r)_*F_l$ on $Y_r$. Note that $Y_r$ is Fano of Picard number one but with an isolated canonical singularity at $\varphi_r(E)$. Hence the result of \cite{Pr2} on terminal Fano threefold of Picard number one does not apply. Intersecting with a covering family of curves $\{C_t\}$ of the minimal length on the canonical Fano variety $Y_r$,  we get $aF_{Y_r}\cdot C_t\leq 4$ by \cite{CT} on the length of extremal curves of klt threefolds. However, we do not conclude that $a\leq 4$ as the divisor $F_{Y_r}$ is not necessarily Cartier, cf. Example \ref{ex}. The index of $K_{Y_r}$ at $P=\varphi_r(E)$ is at most $6$ by \cite{Kaw}. However, this bound is not sufficient for our purpose. 
\end{remark}

\begin{example}\label{ex} Consider $f:X=\bP_{\bP^2}(\cO\oplus\cO(3))\ra Z=\bP^2$. Then there is a section $S$ with $\cO_X(1)\sim S$, $S|_S=\cO_{\bP^2}(3)$, and $-K_X=2S$ is nef and big. There is a negative section $S'\cong Z$ with $S'|_{S'}=\cO_{\bP^2}(-3)$, which can be contracted by the anticanonical morphism to get $\varphi_{|-K_X|}:X\ra W\cong\bP(1,1,1,3).$ If $H=f^*\cO_{\bP^2}(1)$, then $-K_X=6H+2S'$ and $-K_W=6H_W$. Here $H_W$ is $\bQ$-Cartier but not Cartier. Note that we have 
	$$-K_X^3=8S^3=72\ {\rm or}\ -K_X^3=\underbrace{3\cdot 36H^2\cdot 2S'+3\cdot6H\cdot4S'^2}_{=0}+8S'^3=72.$$
	\end{example}

It is now easy to establish the anticanonical volume bound except one case.
\begin{proposition}\label{3} Suppose that $X$ is a weak $\bQ$-Fano threefold of Picard rank two and we are in Case {\bf (II)}. Then $-K_X^3\leq72$ unless $\dim Z_l=1$ and \mbox{$\dim\varphi_r(E)=0$}.
\end{proposition}
\begin{proof} We claim that 
$$-K_X^3\leq\begin{cases} aK_{X_l}^2\cdot F_l=aK_{F_l}^2\leq 9a,&\ {\rm if}\ \dim Z_l=1;\\  aK_{X_l}^2\cdot H_l\leq 12a, &\ {\rm if}\ \dim Z_l=2,\end{cases}$$
where the last inequality is Lemma \ref{Z2}. By Lemma \ref{gen}, this leads to   
$$-K_X^3\leq\begin{cases} 4\cdot12=48 ,&\ {\rm if}\ \dim Z_l=2,\ \dim\varphi_r(E)=1;\\ 6\cdot12=72, &\ {\rm if}\ \dim Z_l=2,\ \dim\varphi_r(E)=0;\\ 6\cdot9=54  &\ {\rm if}\ \dim Z_l=1,\ \dim\varphi_r(E)=1;\\  9\cdot9=81 ,&\ {\rm if}\ \dim Z_l=1,\ \dim\varphi_r(E)=0.\end{cases}$$ 
Hence $-K_X^3\leq 72$ unless $\dim Z_l=1$ and $\dim\varphi_r(E)=0$. 

To prove the claimed inequalities, suppose that $\dim Z_l=1$ and write $-K_X\equiv aF_X+bE$ as in Lemma \ref{gen}. From Lemma \ref{neg}, we have $h^*F_X-p_l^*F_l=E'\geq0$ and $h^*K_X-p_l^*K_{X_l}=E''\geq0$, where both $E', E''\geq0$ are exceptional over $X$ and $X_l$. Since $K_X=\varphi_r^*K_{Y_r}$ and $K_X^2\cdot E=0$,  we get  
\begin{align*} -K_X^3=&(-K_X)^2\cdot (aF_X+bE)\\ 
				=& aK_X^2\cdot F_X\\
				=&a(h^*K_X)^2\cdot (p_l^*F_l+E') \\
				=&a(p_l^*K_{X_l}+E'')^2\cdot p_l^*F_l \\	
				=&a(K_{X_l}^2\cdot F_l+2p_l^*K_{X_l}\cdot p_l^*F_l.E''+E''^2\cdot p_l^*F_l)\\
				\leq&aK_{X_l}^2\cdot F_l,
\end{align*}
where we have used projection formula $(h^*K_X)^2\cdot E'=0=p_l^*K_{X_l}\cdot p_l^*F_l\cdot E''$, and in the last inequality $E''^2\cdot p_l^*F_l=(E''|_{p_l^*F_l})^2\leq 0$.  Indeed, take $S_l\in |p_l^*(mF_l)|$ for $m\gg0$ a general smooth surface of this base point free linear system, then $E''|_{S_l}$ is an exceptional curve for the birational morphism $p_l|_{S_l}$ and hence $m^2p_l^*F_l\cdot E''^2=(E''|_{S_l})^2\leq0$ by \cite[Corollary 2.7]{B}. 
A similar computation shows that $-K_X^3\leq aK_{X_l}^2\cdot H_l$ when $\dim Z_l=2$.
\end{proof}

\section{Optimal examples}\label{opt}
Except when $\dim Z_l=1$ and $\dim\varphi_r(E)=0$ in Case ({\bf II}) as in Proposition \ref{3}, we finish the proof of Main Theorem by characterizing weak $\bQ$-Fano threefolds with optimal anticanonical volume 72.
\begin{proposition}\label{opt} Let $X$ be a weak $\bQ$-Fano threefold of $\rho(X)=2$. Assume that we are in Case ({\bf II}), and $\dim Z_l=2$ or \mbox{$\dim\varphi_r(E)=1$}. If \mbox{$-K_X^3=72$}, then 
$X\cong\bP_{\bP^2}(\cO_{\bP^2}\oplus\cO_{\bP^2}(3)).$
\end{proposition}
\begin{proof} By Proposition \ref{3}, $-K_X^3=72$ is only possible if $\dim Z_l=2$ and $\dim\varphi_r(E)=0$, and there is a diagram 
\begin{center}
	\begin{tikzcd} 
		X_l\arrow[d,"f_l"] && X\arrow[ll,dashed,swap,"\chi_l"]\arrow[r,"\varphi_r"]   &Y_r\\
		Z_l   &&E\arrow[u,hook]\arrow[r]&\varphi_r(E)=pt\arrow[u,hook].
	\end{tikzcd} 
\end{center}

As $\dim Z_l=2$, from Proposition \ref{3} we have $-K_X^3= aK_{X_l}^2\cdot H_l=6\cdot 12=72,$ i.e., $a=6$ and $K_{X_l}^2\cdot H_l=12$. Trace back all the computation, we must have $b=2$ and $E\cdot \Gamma=-3$. On the other hand, 
$K_{X_l}^2\cdot H_l=-(4K_{Z_l}+\Delta)\cdot C=12$ implies that $(Z_l,C)\cong(\bP^2,{\rm line})$ and $\Delta=0$.  

Since $2=-K_{X_l}\cdot l=(6H_l+2E_l)\cdot l=2E_l\cdot l$ for a general fiber $l$ of the conic bundle $f_l:X_l\ra Z_l$ and $E_l\cdot l\in\bZ$ as the intersection can be taken to be over the smooth part of $X_l$, we have $E_{X_l}\cdot l=1$. It follows that $E_{X_l}\ra \bP^2$ is a bijective projective morphism and hence $E_{X_l}\cong\bP^2$ by Zariski main theorem \cite{H}. We claim that $\chi_l$ is identity. Suppose not and let $\Gamma'$ be a flipping curve of the last flip $X'\dra X_l$, then 
$$0>-K_{X_l}\cdot\Gamma'=6H_l\cdot\Gamma'+2E_l\cdot\Gamma'\geq2E_l\cdot\Gamma',$$
and hence $\Gamma'\subseteq E_l$. But then the flip has codimension one exceptional locus as $\NE(E_l)=\NE(\bP^2)$ is one dimensional, which is absurd. Hence $X=X_l$ and it has a conic bundle structure $f:X\ra \bP^2$. This indeed is a $\bP^1$-bundle over $U:=\bP^2\backslash f(\Sing(X))$ as $X_U=f^{-1}(U)$ is now smooth and $\Delta=0$, cf. Proposition \ref{conic}. In particular, for $L=\bP^1$ a projective line in $U\subseteq\bP^2$, $H=f^*L$ is a $\bP^1$-bundle over $\bP^1$ and hence 
$$8=K_H^2=(K_X+H)|_H^2=(5H|_H+2E|_H)^2=20H^2\cdot E+4E^2\cdot H.$$
As $H^2\cdot E=l\cdot E=1$, we get $E^2\cdot H=E|_H^2=-3$, i.e., $E|_H$ is the unique negative section on $H$, and  $H\cong\bP_{\bP^1}(\cO_{\bP^1}\oplus\cO_{\bP^1}(3)).$\footnote{From this and $12=K_X^2\cdot H$, one can conclude $K_X^2\cdot E=K_X\cdot E^2=0\ {\rm and}\  E^3=9.$} 

We now finish the proof. The sheaf $V:=(f_*\cO_X(E))^{\vee\vee}$ is reflexive and hence locally free of rank two on $\bP^2$. The section $E$ over $\bP^2$ corresponds to a rank one free quotient of $V$, which indeed splits $V$ as the first cohomology of a line bundle vanishes on $\bP^2$. By restricting to $H$, we see that $\bP_{\bP^2}(V)$ can only be $X':=\bP(\cO_{\bP^2}\oplus\cO_{\bP^2}(3))$. On the other hand, by construction $X_U\cong\bP_U(V|_U)=\bP_{\bP^2}(V)|_U$. Hence $X'$ and $X$ are isomorphic in codimension one. Since $X'$ is a weak Fano variety and hence a Mori dream space by \cite{BCHM}, and $X$ is a small $\bQ$-factorial modification (SQM) of $X'$, we must have $X'\cong X$ as clearly, $X'$ has no nontrivial SQM. 
\end{proof}

The Main Theorem now follows from Section \ref{I}, Proposition \ref{3}, and Proposition \ref{opt}.

\begin{remark}[\cite{Isk}] The canonical Fano threefold $Y=\bP(1,1,4,6)$ satisfies $-K_Y^3=72$ and can be described as the image of the birational map $\varphi_{|-K_X|}$ of the weak Fano threefold $X:=\bP_{\bP^1}(\cO_{\bP^1}\oplus\cO_{\bP^1}(2)\oplus\cO_{\bP^1}(6))$, where $-K_X^3=54$. However, a terminalization of $Y$ must have Picard rank strictly bigger than two and can not be an example of optimal anticanonical volume in our main theorem. Indeed, a toric description of $Y$ can be given by the toric fan  $\<e_1,e_2,e_3,v=-4e_1-6e_2-e_3\>$ in $\bR^3$, for which blowing up at vectors $(-1,-1,0)$ and $(-2,-3,0)$ are crepant and a terminalization of $Y$ must have Picard rank bigger than two. 
\end{remark}

\bibliographystyle{plain}
\bibliography{vol}
\end{document}